\documentclass[12pt]{amsart}

\usepackage{amssymb,latexsym,amsmath}
\input{xypic}
\xyoption{all}

\newcommand{\authorfootnotes}{\renewcommand\thefootnote{\@fnsymbol\c@footnote}}%
\makeatother

\newtheorem{theorem}{Theorem}[section]
\newtheorem{lemma}[theorem]{Lemma}
\newtheorem{proposition}[theorem]{Proposition}

\newtheorem{remark}[theorem]{Remark}

\numberwithin{equation}{section}

\begin{document}

\def\q{\mathfrak{q}}
\def\p{\mathfrak{p}}
\def\l{\mathfrak{l}}
\def\u{\mathfrak{u}}
\def\a{\mathfrak{a}}
\def\b{\mathfrak{b}}
\def\m{\mathfrak{m}}
\def\n{\mathfrak{n}}
\def\c{\mathfrak{c}}
\def\d{\mathfrak{d}}
\def\e{\mathfrak{e}}
\def\k{\mathfrak{k}}
\def\z{\mathfrak{z}}
\def\h{{\mathfrak h}}
\def\gl{\mathfrak{gl}}
\def\sl{\mathfrak{sl}}
\def\bk{\boldsymbol{k}}
\def\bl{\boldsymbol{l}}
\def\Ext{{\rm Ext}}
\def\Hom{{\rm Hom}}
\def\Ind{{\rm Ind}}

\def\res{\mathop{Res}}

\def\GL{{\rm GL}}
\def\SL{{\rm SL}}
\def\SO{{\rm SO}}
\def\O{{\rm O}}

\def\R{\mathbb{R}}
\def\C{\mathbb{C}}
\def\Z{\mathbb{Z}}
\def\Q{\mathbb{Q}}
\def\A{\mathbb{A}}

\def\w{\wedge}

\def\Cat{\mathcal{C}}
\def\HC{{\rm HC}}
\def\HCat{\Cat^\HC}
\def\proj{{\rm proj}}

\def\to{\rightarrow}
\def\To{\longrightarrow}

\def\1{1\!\!1}
\def\dim{{\rm dim}}

\def\th{^{\rm th}}
\def\isom{\approx}

\def\CE{\mathcal{C}\mathcal{E}}
\def\E{\mathcal{E}}

\def\dis{\displaystyle}
\def\f{{\bf f}}                 
\def\g{{\bf g}}
\def\T{{\rm T}}              
\def\omegatil{\tilde{\omega}}  
\def\H{\mathcal{H}}            
\def\Dif{\mathfrak{D}}      
\def\W{W^{\circ}}           
\def\Whit{\mathcal{W}}      
\def\ringO{\mathcal{O}}     
\def\S{\mathcal{S}}      
\def\M{\mathcal{M}}      
\def\K{{\rm K}}          
\def\h{\mathfrak{h}} 
\def\N{\mathfrak{N}}    
\def\norm{{\rm N}}       
\def\trace{{\rm Tr}} 
\def\ctilde{\tilde{C}}

\author{Alia Hamieh}
\address[Alia Hamieh]{University of Lethbridge, Department of Mathematics and Computer Science, C526 University Hall, 4401 University Drive, Lethbridge, AB T1K3M4, Canada}
\email{alia.hamieh@uleth.ca}

\author{Naomi Tanabe}
\address[Naomi Tanabe]{Dartmouth College, Department of Mathematics, 6188 Kemeny Hall, Hanover, NH 03755-3551, USA}
\email{naomi.tanabe@dartmouth.edu}

\thanks{Research of both authors was partially supported by Coleman Postdoctoral Fellowships at Queen's University}
\thanks{Research of the first author is currently supported by a PIMS Postdoctoral Fellowship at the University of Lethbridge}
\keywords{ Hilbert modular forms, Rankin-Selberg Convolution, Petersson Trace Formula}
\subjclass[2010]{Primary 11F41, 11F67; secondary 11F30, 11F11}

\title[Determining Hilbert Modular Forms: Level Aspect]
{Determining Hilbert Modular Forms by Central Values of Rankin-Selberg Convolutions: The Level Aspect}

\date{\today}

\begin{abstract}
In this paper, we prove that a primitive Hilbert cusp form $\g$ is uniquely determined by the central values of the Rankin-Selberg $L$-functions $L(\f\otimes\g, \frac{1}{2})$, where $\f$ runs through all primitive Hilbert cusp forms of level $\q$ for infinitely many prime ideals $\q$. This result is a generalization of the work of Luo \cite{luo2} to the setting of totally real number fields.
 \end{abstract}

\maketitle

\setcounter{section}{-1}

\section{Introduction}

An interesting question is to what extent the special values of automorophic $L$-functions determine the underlying automorphic forms. More precisely, several mathematicians have addressed the problem of identifying an automorphic form by the special values of the $L$-function of its twists by a family of automorphic forms (on $\GL_{1}$ or $\GL_{2}$). In particular, Luo and Ramakrishnan proved in an important paper \cite{luo-ramakrishnan} that two primitive cusp forms $g$ and $g'$  (on $\GL_2(\Q)$) must be equal if the special values $L(g\otimes\chi_{d},\frac{1}{2})$ and $L(g'\otimes\chi_{d},\frac{1}{2})$ are equal (up to a constant) for all but finitely many quadratic characters $\chi_{d}$. This result has been generalized by Chinta and Diaconu \cite{chinta-diaconu} to $\GL_{3}$-forms. It has also been genralized by Li \cite{li} to self-contragredient automorphic cuspidal representations of $\GL_2$ over any number field. 

Choosing for the twisting family the set of primitive forms of fixed even weight and infinitely varying level, Luo \cite{luo2} proved the following. Let $g$ and $g'$ be primitive cusp forms (over $\Q$) of even weights and general levels. Let $c$ be a constant and $k$ be a positive integer. If there exit infinitely many primes $p$ such that 
$$L\left(f\otimes g,\frac{1}{2}\right)=cL\left(f\otimes g',\frac{1}{2}\right)$$ 
for all primitive cusp forms $f$ of weight $2k$ and level $p$, then we have $g=g'$.

Ten years later, Ganguly, Hoffstein and Sengupta proved in \cite{ganguly-hoffstein-sengupta} an analogous result upon twisting by the family of primitive cusp forms of level 1 and weight $2k$ as $k$ tends to infinity. A similar result for determining modular forms of general level can be found in \cite{zhang}.

It is our purpose in this paper to extend Luo's approach in \cite{luo2} to the setting of an arbitrary totally real number field $F$. If the narrow class number of $F$ is greater than 1, one immediately confronts a number of difficulties, the most important of which is the lack of an action of Hecke operators under which the space of classical Hilbert modular forms (of given weight and level) is stable. In order to overcome this difficulty, we consider the larger space of ad\`elic Hilbert modular forms which unlike its classical counterpart, is invariant under the action of Hecke operators (see Section~\ref{sec:hmf}). 

In this paper, we prove the following theorem (the reader is referred to Sections 1 and 2 for notation and terminology).
\begin{theorem}\label{main2}
Let $\g\in S_{\bl}^{\mathrm{new}}(\n)$ and $\g'\in S_{\bl'}^{\mathrm{new}}(\n')$ be normalized Hilbert eigenforms, with the weights $\bl$ and $\bl'$ being in $2\mathbb{N}^n$. Let $\bk\in2\mathbb{N}^{n}$ be fixed, and suppose that there exist infinitely many prime ideals $\q$ such that 
$$L\left(\f\otimes \g,\frac{1}{2}\right)=L\left(\f\otimes \g',\frac{1}{2}\right)$$ for all normalized Hilbert eigenforms $\f\in S_{\bk}^{\mathrm{new}}(\q)$. Then $\g=\g'$.
\end{theorem}

The proof of this theorem can be found in Section~\ref{sec:proof}. The idea is to show that the Fourier coefficients $C_{\g}(\p)$ and $C_{\g'}(\p)$ are equal for all but finitely many prime ideals $\p$. The result then follows by the strong multiplicity one theorem (cf. \cite[Chapter~3]{bump} and \cite{miyake}). We accomplish this by appealing to the technique used in \cite{ganguly-hoffstein-sengupta}, \cite{luo2} and \cite{luo-ramakrishnan}. Roughly speaking, our target is to express the coefficient $C_{\g}(\p)$ in terms of the central values $L\left(\f\otimes\g,\frac{1}{2}\right)$ up to an error term $E$. See Proposition~\ref{main-proposition}. Indeed, 
we show that
$$\sum_{\f}L\left(\f\otimes \g,\frac12\right)C_\f(\p)\omega_\f=\frac{C_{\g}(\p)}{\sqrt{\norm(\p)}}M + E,$$
where the sum is taken over all primitive forms of weight $\bk\in2\mathbb{N}^{n}$ and level $\q$ with $\norm(\q)$ being sufficiently large. 
This is done in Section~\ref{sec:petersson} upon applying an approximate functional equation established in Section~\ref{sec:approx}, together with a Petersson-type trace formula described in Section~\ref{trace-formula}. Finally, we prove that $M\sim C\log(\norm(\q))$ and $E=O(1)$ as $\norm(\q)\to\infty$. 
This computation is shown in Section~\ref{sec:average}.

As will be obvious in later sections, the proof entails several complications and subtleties arising from the technical nature of ad\`elic Hilbert modular forms over a totally real number field $F$ and the infinitude of the group of units in $F$. For example, in dealing with the error term $E$, we encounter a summation of the Kloosterman sum weighted by a product of the classical $J$-Bessel functions:  \begin{equation}\label{eqn:intro}\sum_{c\in\c^{-1}\q\backslash\{0\}/\ringO_{F}^{\times +}}\sum_{\eta\in\ringO_{F}^{\times+}}\frac{{\mathit Kl}(\nu,\a; \xi,\b;c\eta,\c)}{|\norm(c)|}\prod_{j=1}^{n}J_{k_j-1}\left(\frac{4\pi\sqrt{\nu_j\xi_j\left[\a\b\c^{-2}\right]_j}}{\eta_{j}|c_j|}\right).\end{equation} The new feature here is the sum over totally positive units, which originates from the application of a Petersson-type trace formula for Hilbert modular forms due to Trotabas \cite{trotabas}. In order to estimate (\ref{eqn:intro}), we employ a trick due to Luo \cite{luo1} which amounts to bounding the values of the classical $J$-Bessel function in such a way that the sum over units can be factored out as  $$\sum_{\eta\in\ringO_{F}^{\times+}}\prod_{\eta_{j}<1}\eta_{j}^{\delta},$$ which is convergent for all $\delta>0$ (see \cite[p.\:136]{luo1}). The reader is referred to Section~\ref{sec:error-term} for the details.

Another interesting problem to consider is the weight aspect analogue of the present work. 
In a separate paper \cite{hamieh-tanabe}, the authors obtain a result in this direction by following the line of argument developed in \cite{ganguly-hoffstein-sengupta} and \cite{zhang}.

\section{Notations and Preliminaries}\label{sec:background}
 \subsection{The Base Field}\label{sec:base-field} Let $F$ be a totally real number field of degree $n$ over $\mathbb{Q}$, and let $\ringO_{F}$ be its ring of integers. The real embeddings of $F$ are denoted by $\sigma_{j}:x\mapsto x_{j}:=\sigma_{j}(x)$ for $j=1,\dots,n$. Any element $x$ in $F$ may therefore be identified with the $n$-tuple $\boldsymbol{x}=(x_{1},\dots,x_{n})\in\mathbb{R}^{n}$. We say $x$ is totally positive and write $x\gg0$ if $x_{j}>0$ for all $j$, and for any subset $X\subset F$, we put $X^{+}=\{x\in X: x\gg0\}$. 

The trace and the norm of an element $x$ in the field extension $F/\mathbb{Q}$ are denoted by $\trace(x)$ and $\norm(x)$ respectively. The absolute norm of an integral ideal $\a$ is $\norm(\a)= [\ringO_{F}:\a]$. Notice that for a principal ideal $(\alpha)=\alpha\ringO_{F}$ of $\ringO_{F}$, we have $\norm\left((\alpha)\right)=|\norm(\alpha)|$. The absolute norm defined as such can be extended by multiplicativity to the group $I(F)$ of fractional ideals of $F$. The different ideal of $F$ and the discriminant of $F$ over $\mathbb{Q}$ are denoted by $\Dif_{F}$ and $d_{F}$ respectively, and we have the identity $\norm(\Dif_{F})=|d_{F}|$. 

Recall that the narrow class group of $F$ is the quotient group $Cl^{+}(F)=I(F)/P^{+}(F)$, where $P^{+}(F)$ is the group of principal ideals $(\alpha)$ with $\alpha\in F^{\times +}$. It is a finite group, and we denote its cardinality by $h_{F}^{+}$. We fix once and for all a system of representatives $\{\a\}$ of $Cl^{+}(F)$ in $I(F)$.  Given two fractional ideals $\a$ and $\b$, we write $\a\sim\b$ if there exists $\xi\in F^{\times +}$ such that $\a=\xi\b$ in which case we use the notation $[\a\b^{-1}]$ to denote the element $\xi$. Needless to say, if $\a\sim\b$, then $[\a\b^{-1}]\in F^{\times +}$ is only unique up to muliplication by totally positive units in $\ringO_{F}$. Let us now recall the following lemma (\cite[Lemma~2.1]{trotabas}) which we use in Section~\ref{sec:error-term}.
\begin{lemma}\label{normlemmatrotabas}
There exist constants $C_{1}$ and $C_{2}$ depending only on $F$ such that $$\forall \xi\in F,\exists\epsilon\in\ringO_{F}^{\times +},\forall j\in\{1,\dots,n\}: \quad C_{1}|\norm(\xi)|^{1/n}\leq|(\epsilon\xi)_{j}|\leq C_{2}|\norm(\xi)|^{1/n}.$$
\end{lemma}
Frequently in this paper, we make use of the multi-index notation. For example, given $n$-tuples $\boldsymbol{y}$ and $\boldsymbol{z}$ and a scalar $a$, we set $$\Gamma(\boldsymbol{z})=\prod_{j=1}^{n}\Gamma(z_{j}),\quad a^{\boldsymbol{z}}=a^{\sum_{j=1}^{n}z_{j}},\quad \boldsymbol{y}^{\boldsymbol{z}}=\prod_{j=1}^{n}y_{j}^{z_{j}}.$$Moreover, for any $x\in F$ and $\boldsymbol{u}\in\mathbb{N}^{n}$, we have $\displaystyle{J_{\boldsymbol{u}}(\boldsymbol{x})=\prod_{j=1}^{n}J_{u_j}(x_{j})}$ ,where $J_{u_{j}}$ is the classical $J$-Bessel function (see Section~\ref{sec:trace-formula}).

\subsection{Hilbert Modular Forms}\label{sec:hmf}
 We now fix some notation pertaining to the space of ad\`elic Hilbert modular forms.  To this end, we closely follow the exposition in \cite{trotabas}. Let $\mathbb{A}_{F}$ be the ad\`ele ring of $F$. For a place $v$ of $F$, we denote by $F_{v}$ the completion of $F$ at $v$ and by $\ringO_{F_{v}}$ the local ring of integers when $v<\infty$. Let $F_{\infty}=\prod_{v\in S_{\infty}}F_{v}$, where $S_{\infty}$ is the set of infinite places of $F$. In what follows, we make the identifications $$F_{\infty}=\mathbb{R}^{n},\quad \GL_{2}^{+}(F_{\infty})=\GL_{2}^{+}(\mathbb{R})^{n},\quad \SO_{2}(F_{\infty})=\SO_{2}(\mathbb{R})^{n},$$
where the superscript ``$+$" means the subgroup consisting of elements with totally positive determinants.
In particular, each $r\in \SO_{2}(F_{\infty})$ can be expressed as $$r(\boldsymbol{\theta})=\left(r(\theta_{j})\right)_{0\leq j\leq n}=\left(\left[ \begin{array}{cc}
\cos\theta_{j} & \sin\theta_{j}  \\
-\sin\theta_{j} & \cos\theta_{j} \end{array} \right]\right)_{0\leq j\leq n},$$ in which case we denote the $n$-tuple $(\theta_{1},\cdots,\theta_{n})$ by $\boldsymbol{\theta}$. 

Given an ideal $\n\subset\ringO_{F}$ and a non-archimedean place $v$ in $F$, we define the subgroup $K_{v}(\n)$ of $\GL_{2}(F_{v})$ as $$K_{v}(\n)=\left\{\left[ \begin{array}{cc}
a & b  \\
c & d \end{array} \right]\in \GL_{2}(\ringO_{F_{v}}): c\in \n \ringO_{F_{v}}\right\}.$$ Then we set $$K_{0}(\n)=\prod_{v<\infty}K_{v}(\n).$$ 

By an ad\`elic Hilbert cusp form $\f$ of weight $\bk\in2\mathbb{N}^{n}$ and level $\n$, we mean a complex-valued function on $\GL_{2}(\mathbb{A}_F)$ which satisfies the following properties (\cite[Definition~3.1]{trotabas}).
\begin{enumerate}
\item The transformation property $\f(\gamma zgr(\boldsymbol{\theta})u)=\f(g)\exp(i\boldsymbol{k}\boldsymbol{\theta})$ holds for all \\ 
$(\gamma, z,g,r(\boldsymbol{\theta}),u)\in \GL_{2}(F)\times\mathbb{A}_{F}^{\times}\times \GL_{2}(\mathbb{A}_{F}) \times \SO_{2}(F_{\infty})\times K_{0}(\n)$.
\item Viewed as a smooth function on $\GL_{2}^{+}(F_{\infty})$, $\f$ is an eigenfunction of the Casimir element $\boldsymbol{\Delta}:=(\Delta_{1},\cdots,\Delta_{n})$ with eigenvalue $\displaystyle{\prod_{j=1}^{n}\frac{k_{j}}{2}\left(1-\frac{k_{j}}{2}\right)}$.
\item We have $\displaystyle{\int_{F\backslash\mathbb{A}_{F}}\f\left(\left[ \begin{array}{cc}
1 & x  \\
0 & 1 \end{array} \right]g\right)dx=0}$ for all $g\in \GL_{2}(\mathbb{A}_{F})$ (cuspidality condition).
\end{enumerate} 

Notice that, in our definition, all the forms are understood to have the trivial character.
We denote by $S_{\bk}(\n)$ the space of ad\`elic Hilbert cusp forms of weight $\bk$ and level $\n$. 
It is also worth noting that this space is trivial if $\bk\not\in 2\mathbb{N}^n$. This is why we impose the parity condition, $\bk\in 2\mathbb{N}^n$, on the weight vectors throughout the paper.

\begin{remark}
It follows from the strong approximation theorem for $\GL_{2}$ that an ad\`elic Hilbert cusp form $\f$ can be viewed as an $h_{F}^{+}$-tuple $(f_{1},...,f_{h_{F}^{+}})$ of classical Hilbert modular forms on $\mathbb{H}^{n}$. More details on this correspondence between the ad\`elic setting and the classical setting can be found in most references on the topic of Hilbert modular forms among which are \cite[Chapter 1, 2]{garrett}, \cite[Section 4]{JRMS-2011}, and \cite[Section 2]{shimura-duke}.
\end{remark}

The subspace of oldforms in $S_{\bk}(\n)$ is denoted by $S_{\bk}^{\mathrm{old}}(\n)$. Roughly speaking, this is the subspace of cusp forms ``obtained'' from cusp forms of lower levels. The orthogonal complement of $S_{\bk}^{\mathrm{old}}(\n)$ with respect to the inner product $$\left<\f,\bf{h}\right>_{S_{\bk}(\n)}=\int_{\GL_{2}(F)\mathbb{A}_{F}^{\times}\backslash \GL_{2}(\mathbb{A}_{F})/K_{0}(\n)}\f(g)\overline{{\bf{h}}(g)}\;dg$$  is referred to as the space of newforms and is denoted by $S_{\bk}^{\mathrm{new}}(\n)$. 

For a Hilbert cusp form $\f$, let $\left\{C(\nu,\a,\f)\right\}_{\nu\in\a^{-1}\Dif_{F}^{-1}}$ be the coefficients given by the Fourier expansion: 
$$\f\left(g\left[ \begin{array}{cc}
\mathrm{id}(\a) & 0  \\
0 & 1 \end{array} \right] \right)=\sum_{\substack{\nu\in\a^{-1}\Dif_{F}^{-1}\\\nu\gg0}}\frac{C(\nu,\a,\f)}{\norm(\nu\a\Dif_{F})^{\frac{1}{2}}}W_{\infty}^{0}\left(\left[\begin{array}{cc}
\boldsymbol{\nu} & 0  \\
0 & 1 \end{array} \right] g\right),\quad\quad g\in \GL_{2}^{+}(F_{\infty}).$$
We mention here that $\mathrm{id}(\a)$ is the idele of $F$ associated with the ideal $\a$, and that $W_{\infty}^{0}$ is the new vector in the Whittaker model of the discrete series representation $\bigotimes_{j}\mathcal{D}(k_{j}-1)$ of $\GL_{2}(F_{\infty})$ (restricted to $\GL_{2}^{+}(F_{\infty})$). In fact, $W_{\infty}^{0}(g)$ for $g\in \GL_{2}^{+}(F_{\infty})$ can be calculated as follows. By the Iwasawa decomposition, we know that $g$ can be uniquely expressed as $$g=\left[ \begin{array}{cc}
\boldsymbol{z} & 0  \\
0 & \boldsymbol{z} \end{array} \right]\left[ \begin{array}{cc}
1 & \boldsymbol{x}  \\
0 & 1 \end{array} \right]\left[ \begin{array}{cc}
\boldsymbol{y} & 0  \\
0 & 1 \end{array} \right]r(\boldsymbol{\theta}),$$ 
with $\boldsymbol{z},\boldsymbol{y}\in F_{\infty}^{\times+}$, $\boldsymbol{x}\in F_{\infty}$ and $r(\boldsymbol{\theta})\in \SO_{2}(F_{\infty})$. Then we have $$W_{\infty}^{0}(g)=\boldsymbol{y}^{\bk/2}\exp\left(2i\pi (\boldsymbol{x}+i\boldsymbol{y})\right)\exp\left(i\bk\boldsymbol{\theta}\right)$$

Let $\m\subset\ringO_{F}$ be an ideal. The Fourier coefficient of $\f$ at $\m$ is denoted by $C_{\f}(\m)$ and defined as follows. We write $\m=\nu\a$ for some narrow ideal class representative $\a$ and some totally positive element $\nu\in\a^{-1}$. Then we set \begin{equation}\label{eqn:fc}C_{\f}(\m)=C(\nu,\a\Dif_{F}^{-1},\f).\end{equation} We say $\f$ is normalized if $C_{\f}(\ringO_{F})=1$. 

Much like the classical setting over $\mathbb{Q}$, one can define an action of Hecke operators $\{T_{\m}\}_{\m\subset\ringO_{F}}$ on the space $S_{\bk}(\n)$ (see, for example, \cite[Section 4]{JRMS-2011}). A Hilbert cusp form $\f$ is said to be primitive if it is a normalized newform and a common eigenfunction for all the Hecke operators. It is known that, for such a form $\f$, the coefficients $C_\f(\m)$ coincide with the Hecke eigenvalues for $T_\m$ for all $\m$ (see \cite[p.\:650]{shimura-duke}).
We denote by $\Pi_{\bk}(\n)$ the (finite) set of all primitive forms of weight $\bk$ and level $\n$. It follows from a standard result of Shimura \cite{shimura-duke} that the coefficients $C_{\f}(\m)$ are real for all $\m$ if $\f$ is a primitive form (because the trivial character is assumed).
It is also worth mentioning that $\Pi_{\bk}(\n)$ can be viewed as the set of all cuspidal automorphic representations $\pi$ of conductor $\n$ with the trivial central character such that $\pi_{\infty}=\bigotimes_{j}\mathcal{D}(k_{j}-1)$.

\subsection{Petersson Trace Formula}\label{sec:trace-formula}
Crucial to our work is a Petersson trace formula for Hilbert modular forms due to Trotabas (see \cite[Theorem~ 5.5, Proposition~6.3]{trotabas}), which we state below.
\begin{proposition}\label{trace-formula}
Let $\q$ be an integral ideal in $F$. Let $\a$ and $\b$ be fractional ideals in $F$. For $\alpha\in\a^{-1}$ and $\beta\in\b^{-1}$, we have \begin{align*} &\sum_{\f\in H_{\bk}(\q)}\frac{\Gamma(\bk-\boldsymbol{1})}{(4\pi)^{\bk-\boldsymbol{1}}|d_{F}|^{1/2}\left<\f,\f\right>_{S_{\bk}(\q)}}C_{\f}(\alpha\a)\overline{C_{\f}(\beta\b)}\\
&\hspace{.6in}=\1_{\alpha\a=\beta\b}+C\sum_{\substack{\mathfrak{c}^{2}\sim\a\b\\c\in\c^{-1}\q\backslash\{0\}\\\epsilon\in\mathcal{O}_{F}^{\times+}/\mathcal{O}_{F}^{\times2}}}\frac{{\mathit Kl}(\epsilon\alpha,\a; \beta,\b;c,\c)}{\norm(c\c)}J_{\bk-\boldsymbol{1}}\left(\frac{4\pi\sqrt{\boldsymbol{\epsilon}\boldsymbol{\alpha}\boldsymbol{\beta}\boldsymbol{\left[\mathfrak{a}\mathfrak{b}\mathfrak{c}^{-2}\right]}}}{|\boldsymbol{c}|}\right),\end{align*}
where $\dis{C=\frac{(-1)^{\bk/2}(2\pi)^{n}}{2|d_{F}|^{1/2}}}$ and $H_{\bk}(\q)$ is an orthogonal basis for the space $S_{\bk}(\q)$.
\end{proposition}
For convenience, we now recall the definition of the $J$-Bessel function and the Kloosterman sum which appear in the proposition above. The $J$-Bessel function is defined via the Mellin-Barnes integral reprsentation as
$$J_{u}(x)=\int_{(\sigma)}\frac{\Gamma\left(\frac{u-s}{2}\right)}{\Gamma\left(\frac{u+s}{2}+1\right)}\left(\frac{x}{2}\right)^{s}\;ds\quad\quad 0<\sigma<\Re(u).$$ 
It is known that for $u\in\mathbb{N}$ and $x>0$, we have $J_{u}(x)\ll\mathrm{min}(1,x^{u})\leq x^{\lambda}$, if $0\leq\lambda\leq u$ (see \cite[p.\:952]{gradshteyn-ryzhik}). In particular, we have the following bound which we use in Section~\ref{sec:error-term}:
\begin{equation}\label{J-Bessel0}
J_{u}(x)\ll x^{1-\delta},\quad\quad \text{if }\;0\leq\delta\leq1.\end{equation} 
 As for the Kloosterman sum, it is defined as follows. Given two fractional ideals $\a$ and $\b$ , let $\c$ be an ideal such that $\c^{2}\sim\a\b$. For $\nu\in\a^{-1}$, $\xi\in\b^{-1}$ and $c\in\c^{-1}\q$, the Kloosterman sum $\mathit{Kl}(\nu,\a;\xi,\b;c,\c)$ is given by 

\[{\mathit Kl}(\nu,\a; \xi,\b;c,\c)=\sum_{x\in \left(\a\Dif_{F}^{-1}\c^{-1}/\a\Dif_{F}^{-1}c\right)^{\times}}\exp\left(2\pi i\trace\left(\frac{\nu x+\xi\left[\a\b\c^{-2}\right]\overline{x}}{c}\right)\right).\]
Here $\overline{x}$ is the unique element in $\left(\a^{-1}\Dif_{F}\c/\a^{-1}\Dif_{F}c\c^{2}\right)^{\times}$ such that $x\overline{x}\equiv 1\mod c\c$. The reader is refered to Section 2.2 and Section 6 in \cite{trotabas} for more details on this construction. The Kloosterman sum satisfies the following bound (Weil bound): 
\begin{equation}\label{weilbound}
|\mathit{Kl}(\alpha,\n;\beta,\m;c,\c)|\ll_{F}\norm\left(((\alpha)\n,(\beta)\m,(c)\c)\right)^{\frac{1}{2}}\tau((c)\c)\norm(c\c)^{\frac{1}{2}},\end{equation} 
where $(\a,\b,\c)$ is the g.c.d of the ideals $\a$, $\b$, $\c$, and $\tau(\m)=|\{\d\subset \ringO_{F}:\m\d^{-1}\subset\ringO_{F}\}|$ for any integral ideal $\m$. Another useful fact is the well-known estimate: for all $\epsilon>0$, we have
\begin{equation}\label{eqn:tau}
\tau(\n)\ll_{\epsilon}\norm(\n)^{\epsilon}.
\end{equation}

\section{Rankin-Selberg Convolution}
In this section, we recall the construction of Rankin-Selberg convolutions of two Hilbert modular forms following Shimura~\cite[Section~4]{shimura-duke}. It should be noted, however, that our normalization differs slightly from what Shimura uses. Let ${ \f\in S_{\bk}^{\mathrm{new}}(\q)}$ and  ${\g\in S_{\boldsymbol{l}}^{\mathrm{new}}(\n)}$ be primitive forms, where we assume that $\q$ and $\n$ are coprime integral ideals. 
The $L$-series for the associated Rankin-Selberg convolution is defined as \begin{equation}\label{eqn:rankin-selberg}L(\f\otimes \g,s)=\zeta_{F}^{\n\q}(2s)\sum_{\m\subset \ringO_{F}}\frac{C_{\f}(\m)C_{\g}(\m)}{\norm(\m)^{s}},\end{equation}
where the coefficients $C_{\f}(\m)$ and $C_{\g}(\m)$ are as defined in (\ref{eqn:fc}) and 
$$\zeta_{F}^{\n\q}(2s)=\zeta_{F}(2s)\prod_{\substack{\l|\n\q \\ \l \,:\,\mathrm{prime}}}(1-\norm(\l)^{-2s})=\sum_{d=1}^{\infty}\frac{a_{d}^{\n\q}}{d^{2s}}.$$
Here $a_d^{\n\q}$ denotes the number of ideals in $\ringO_F$ that are coprime to $\n\q$ and whose norm is $d$. 
The Ramanujan bound for Hilbert modular forms (proven by Blasius \cite{blasius}) asserts that for all $\epsilon>0$, we have
\begin{equation}\label{eqn:ramanujan-bound}
C_{\f}(\m)\ll_{\epsilon}\norm(\m)^{\epsilon}\quad\text{and}\quad C_{\g}(\m)\ll_{\epsilon}\norm(\m)^{\epsilon}.
\end{equation}
In view of this bound, it is easy to see that the series (\ref{eqn:rankin-selberg}) is absolutely convergent for $\Re(s)>1$. Notice that we can write $$L(\f\otimes \g,s)=\sum_{m=1}^\infty\frac{b_{m}^{\n\q}(\f\otimes\g)}{m^s},$$
with
$$ b_{m}^{\n\q}(\f\otimes\g)=\sum_{d^2|m}\left(a_d^{\n\q}\sum_{\norm(\m)=m/d^2}C_\f(\m)C_\g(\m)\right).$$
Let
$$\Lambda(\f\otimes \g,s)=\norm(\Dif_{F}^{2}\n\q)^{s}L_{\infty}(\f\otimes\g,s)L(\f\otimes \g,s),$$where $$L_{\infty}(\f\otimes\g,s)=\prod_{j=1}^n(2\pi)^{-2s-\mathrm{max}\{k_{j}, l_j\}}\Gamma\left(s+\frac{|k_j-l_j|}{2}\right)\Gamma\left(s-1+\frac{k_j+l_j}{2}\right).$$ Then $\Lambda(\f\otimes \g,s)$ admits an analytic continuation to $\mathbb{C}$ as an entire function (unless $\f=\g$) and satisfies the functional equation (see \cite{prasad-ramakrishnan})\begin{equation}\label{functional-equation}\Lambda(\f\otimes\g,s)=\Lambda(\f\otimes \g,1-s).\end{equation}

\section{Proof of Main Result}\label{sec:proof}
Let $\g\in S_{\boldsymbol{l}}^{\mathrm{new}}(\n)$ be a primitive form, and let $\p$ be either $\ringO_{F}$ or a prime ideal. Let $\q$ be a prime ideal that is relatively prime to $\n$ and $\p$. The main object of interest in this paper is the twisted first moment 
\begin{equation}\label{eqn:first-moment}\sum_{\f\in\Pi_{\bk}(\q)}L\left(\f\otimes \g,\frac12\right)C_\f(\p)\omega_\f,\end{equation} 
where $\Pi_{\boldsymbol{k}}(\q)$ is the set of all primitive forms of weight $\boldsymbol{k}$ and level $\q$, and $w_{\f}$ is defined as 
\begin{equation}\label{eqn:weight}
w_{\f}=\frac{\Gamma(\bk-\boldsymbol{1})}{(4\pi)^{\bk-\boldsymbol{1}}|d_{F}|^{1/2}\left<\f,\f\right>_{S_{\bk}(\q)}}.
\end{equation} 

The following proposition, which we prove in Section~\ref{sec:approx} through Section~\ref{sec:average}, gives an asymptotic formula for this moment in the level aspect ($\norm(\q)\to\infty$). 
\begin{proposition}\label{main-proposition}
Let $\g\in S_{\boldsymbol{l}}^{\mathrm{new}}(\n)$ be a primitive form, and let $\p$ be either $\ringO_{F}$ or a prime ideal. For all prime ideals $\q$ with $\norm(\q)$ sufficiently large, we have 
$$\sum_{\f\in\Pi_{\bk}(\q)}L\left(\f\otimes \g,\frac12\right)C_\f(\p)\omega_{\f}=\frac{C_{\g}(\p)}{\sqrt{\norm(\p)}}\gamma_{-1}(F)A_{\n}\log(\norm(\q))+O(1),$$ 
where $\displaystyle{\gamma_{-1}(F)=2\res_{u=0}\zeta_{F}(2u+1)}$ and $\displaystyle{A_{\n}=\prod_{\substack{\l|\n \\ \l \,:\,\mathrm{prime}}}(1-\norm(\l)^{-1})}$.
\end{proposition}

In what follows, we prove Theorem~\ref{main2} assuming Proposition~\ref{main-proposition}. Let $\g\in S_{\boldsymbol{l}}^{\mathrm{new}}(\n)$ and $\g\in S_{\boldsymbol{l'}}^{\mathrm{new}}(\n')$ be primitive forms. Let $\bk\in2\mathbb{N}^{n}$ be fixed and suppose that there exist infinitely many prime ideals $\q$ such that $L(\f\otimes \g, 1/2)=L(\f\otimes \g',1/2)$ for all $\f\in\Pi_{\bk}(\q)$. Applying Proposition~\ref{main-proposition} with $\p=\ringO_{F}$, we see that
$A_{\n}=A_{\n'}$. We then apply the proposition with $\p$ being any prime ideal not dividing $\n\n'$ to get $C_{\g}(\p)=C_{\g'}(\p)$. Hence, the Hecke eigenvalues of $\g$ and $\g'$ for $T_\p$ are equal (see the last paragraph of Section~\ref{sec:hmf}), and it follows by the strong multiplicity one theorem that $\g=\g'$.

\section{Approximate Functional Equation}\label{sec:approx}
In this section, we establish the approximate functional equation which allows us to write the central value $L\left(\f\otimes\g,\frac{1}{2}\right)$ in terms of a rapidly decaying series built from the Fourier coefficients of $\f$ and $\g$. 

For $X>0$, we put
\[I(s, X)=\frac{1}{2\pi i}\int_{(3/2)}X^u\Lambda(\f\otimes \g,s+u)G(u)\frac{du}{u},\]
where $G(u)$ is a holomorphic function on an open set containing the strip $|\Re(u)|\leq 3/2$. We require that it is bounded and satisfies $G(u)=G(-u)$ and $G(0)=1$. 

By the residue theorem, we have 
\begin{equation*}
\Lambda(\f\otimes \g,s)=\frac{1}{2\pi i}\int_{\left(3/2\right)}X^u\Lambda(\f\otimes \g,s+u)G(u)\frac{du}{u}-\frac{1}{2\pi i}\int_{\left(-3/2\right)}X^u\Lambda(\f\otimes\g,s+u)G(u)\frac{du}{u}.
\end{equation*}
In the last integral, we apply the change of variable $u\mapsto-u$ followed by the functional equation (\ref{functional-equation}) to get
\begin{eqnarray*}
\Lambda(\f\otimes \g,s)=I(s,X)+I(1-s,X^{-1}).
\end{eqnarray*}

On the other hand, $I(s,X)$ can be written as
\begin{align*}
I(s,X)&=\frac{1}{2\pi i}\norm(\Dif_{F}^{2}\n\q)^{s}\sum_{m=1}^\infty \frac{b_{m}^{\n\q}(\f\otimes\g)}{m^s}\int_{\left(3/2\right)}\left(\frac{X\norm(\Dif_{F}^{2}\n\q)}{m}\right)^{u}L_\infty(\f\otimes \g,s+u)G(u)\frac{du}{u}\\
&=\frac{1}{2\pi i}\norm(\Dif_{F}^{2}\n\q)^{s}L_\infty(\f\otimes \g,s)\sum_{m=1}^\infty \frac{b_{m}^{\n\q}(\f\otimes\g)}{m^s}\int_{\left(3/2\right)}\left(\frac{X\norm(\Dif_{F}^{2}\n\q)}{4^n\pi^{2n}m}\right)^u\gamma(s,u)G(u)\frac{du}{u},
\end{align*}
where 
\begin{equation}\label{eqn:gamma} 
\gamma(s,u)=\prod_{j=1}^n\frac{\Gamma\left(s+u+\frac{|k_j-l_j|}{2}\right)\Gamma\left(s+u-1+\frac{k_j+l_j}{2}\right)}{\Gamma\left(s+\frac{|k_j-l_j|}{2}\right)\Gamma\left(s-1+\frac{k_j+l_j}{2}\right)}.\end{equation}
For $y>0$, we define
\begin{equation}\label{eqn:v}
V_s(y):=\frac{1}{2\pi i}\int_{(3/2)}y^{-u}\gamma(s,u)G(u)\frac{du}{u}.
\end{equation}
Then we can write
\[ I(s,X)= \norm(\Dif_{F}^{2}\n\q)^{s}L_\infty(\f\otimes\g,s)\sum_{m=1}^\infty \frac{b_{m}^{\n\q}(\f\otimes\g)}{m^s} V_s\left(\frac{4^n\pi^{2n}m}{X\norm(\Dif_{F}^{2}\n\q)}\right).\]
We also observe that 
\[ I(1-s,X^{-1})=\norm(\Dif_{F}^{2}\n\q)^{1-s}L_\infty(\f\otimes \g,1-s)\sum_{m=1}^\infty\frac{b_{m}^{\n\q}(\f\otimes\g)}{m^{1-s}}V_{1-s}\left(\frac{4^n\pi^{2n}mX}{\norm(\Dif_{F}^{2}\n\q)}\right).\]
Thus,
\begin{align*}
\Lambda\left(\f\otimes \g,\frac12\right)&=I\left(\frac12,X\right)+I\left(\frac12,X^{-1}\right)\\
&=\norm(\Dif_{F}^{2}\n\q)^{1/2}L_\infty\left(\f\otimes \g,\frac12\right)\\
&\hspace{.7in}\times\sum_{m=1}^\infty\frac{b_{m}^{\n\q}(\f\otimes\g)}{\sqrt{m}}\left(V_{1/2}\left(\frac{4^n\pi^{2n}m}{X\norm(\Dif_{F}^{2}\n\q)}\right)+V_{1/2}\left(\frac{4^n\pi^{2n}mX}{\norm(\Dif_{F}^{2}\n\q)}\right)\right).
\end{align*}
Finally, taking $X=1$ leads us to the following proposition.
\begin{proposition}\label{approximate-functional-equation}
Let $G(u)$ be a holomorphic function on an open set containing the strip $|\Re(u)|\leq 3/2$ and bounded therein, satisfying $G(u)=G(-u)$ and $G(0)=1$. Then we have 
\[ L\left(\f\otimes \g,\frac12\right)=2\sum_{m=1}^\infty\frac{b_{m}^{\n\q}(\f\otimes\g)}{\sqrt{m}}V_{1/2}\left(\frac{4^n\pi^{2n}m}{\norm(\Dif_{F}^{2}\n\q)}\right),\]
where $V_{1/2}(y)$ is defined as in (\ref{eqn:v}).

Moreover, the derivatives of $V_{1/2}(y)$ satisfy 
\begin{equation}\label{eqn:V-estimate1}y^{a}V_{1/2}^{(a)}(y)\ll\left(1+\frac{y}{\prod_{j=1}^{n}k_{j}^{2}}\right)^{-A}\end{equation}
and
\begin{equation}\label{eqn:V-estimate2}
y^{a}V_{1/2}^{(a)}(y)=\delta_{a}+O\left(\left(\frac{y}{\prod_{j=1}^{n}k_{j}^{2}}\right)^{\alpha}\right)\end{equation}
for some $0<\alpha\leq1$, where $\delta_{0}=1$, $\delta_{a}=0$ if $a>0$ and the implied constants depend on $a$, $A$ and $\alpha$.
\end{proposition}
\begin{proof} The estimates follow from \cite[Proposition~5.4]{iwaniec-kowalski}. \end{proof}

\section{Application of Petersson Trace Formula}\label{sec:petersson}
The point of departure in this work is a twisted first moment of the central critical values $L(\f\otimes \g,1/2)$ where $\g$ is  fixed in $\Pi_{\boldsymbol{l}}(\n)$ and $\f$ varies over $\Pi_{\bk}(\q)$. More precisely, for an ideal $\p$ which is either $\ringO_{F}$ or a prime ideal different from $\q$, we study the weighted harmonic average introduced in (\ref{eqn:first-moment}). Upon applying the  approximate functional equation in Proposition~\ref{approximate-functional-equation}, we obtain
\begin{align*} 
&\sum_{\f\in\Pi_{\bk}(\q)}L\left(\f\otimes \g,\frac12\right)C_\f(\p)\omega_{\f}\\
&\hspace{.5in}= \sum_{\f\in\Pi_{\bk}(\q)} 2\sum_{m=1}^\infty\frac{b_{m}^{\n\q}(\f\otimes\g)}{\sqrt{m}}V_{\frac{1}{2}}\left(\frac{4^{n}\pi^{2n}m}{\norm(\Dif_{F}^{2}\n\q)}\right)C_\f(\p)\omega_{\f}\\
&\hspace{.5in}= 2\sum_{m=1}^\infty \frac{1}{\sqrt{m}}V_{\frac{1}{2}}\left(\frac{4^n\pi^{2n}m}{\norm(\Dif_{F}^{2}\n\q)}\right)\sum_{\f\in\Pi_{\bk}(\q)}\omega_{\f}C_\f(\p)\sum_{d^2|m}a_{d}^{\n\q}\sum_{\norm(\m)=m/d^2}C_\f(\m)C_\g(\m)\\
&\hspace{.5in}= 2\sum_{m=1}^\infty \frac{1}{\sqrt{m}}V_{\frac12}\left(\frac{4^n\pi^{2n}m}{\norm(\Dif_{F}^{2}\n\q)}\right)\sum_{d^2|m}a_d^{\n\q}\sum_{\norm(\m)=m/d^2}C_\g(\m)\sum_{\f\in\Pi_{\bk}(\q)}\omega_{\f}C_\f(\m)C_\f(\p)\\
&\hspace{.5in}=2\sum_{\m\subset\ringO_{F}} \frac{C_{\g}(\m)}{\sqrt{\norm(\m)}}\sum_{d=1}^{\infty}\frac{a_{d}^{\n\q}}{d}V_{\frac12}\left(\frac{4^n\pi^{2n}\norm(\m)d^2}{\norm(\Dif_{F}^{2}\n\q)}\right)\sum_{\f\in\Pi_{\bk}(\q)}\omega_{\f}C_\f(\m)C_\f(\p).
\end{align*}

For an ideal $\m$ of $\ringO_{F}$, we write $\m=\nu\a$ for some narrow ideal class representative $\a$ and $\nu\in(\a^{-1})^{+}\mod \ringO_F^{\times+}$. In particular, we write $\p$ as $\p=\xi\b$ for some fixed ideal $\b$ and $\xi\in(\b^{-1})^{+}\mod \ringO_F^{\times+}$. 
At this point we invoke the formula in Proposition \ref{trace-formula} to get
\begin{align*}
&\sum_{\f\in\Pi_{\bk}(\q)}L\left(\f\otimes \g,\frac12\right)C_\f(\p)\omega_\f\\
&=2\sum_{\{\a\}}\sum_{\nu\in(\a^{-1})^{+}/\ringO_{F}^{\times +}} \frac{C_{\g}(\nu\a)}{\sqrt{\norm(\nu\a)}}\sum_{d=1}^{\infty}\frac{a_{d}^{\n\q}}{d}V_{\frac{1}{2}}\left(\frac{4^n\pi^{2n}\norm(\nu\a)d^2}{\norm(\Dif_{F}^{2}\n\q)}\right)\\
&\times\left\{\1_{\xi\b=\nu\a}+C\sum_{\substack{\c^{2}\sim\a\b\\c\in\c^{-1}\q\backslash\{0\}\\\epsilon\in\ringO_{F}^{\times+}/\ringO_{F}^{\times2}}}\frac{{\mathit Kl}(\epsilon\nu,\a; \xi,\b;c,\c)}{\norm(c\c)}\prod_{j=1}^nJ_{k_j-1}\left(\frac{4\pi\sqrt{\epsilon_j\nu_j\xi_j\left[\a\b\c^{-2}\right]_j}}{|c_j|}\right)-(\mathrm{old \; forms})\right\},
\end{align*} 
where $\displaystyle{(\mathrm{old\; forms})=\sum_{\f\in H_{\bk}^{\mathrm{old}}(\q)}\frac{\Gamma(\bk-\boldsymbol{1})}{(4\pi)^{\bk-\boldsymbol{1}}|d_{F}|^{1/2}\left<\f,\f\right>_{S_{\bk}(\q)}}C_{\f}(\nu\a)\overline{C_{\f}(\xi\b)}}$ and $H_{\bk}^{\mathrm{old}}(\q)$ is an orthogonal basis for the space of oldforms $S_{\bk}^{\mathrm{old}}(\q)$. Hence, we can write 
\begin{equation}\label{eqn:twisted-first-moment}\sum_{\f\in\Pi_{\bk}(\q)}L\left(\f\otimes \g,\frac12\right)C_\f(\p)\omega_\f=M_{\p}^{\g}(\bk,\q)+E_{\p}^{\g}(\bk,\q)-E_{\p}^{\g}(\bk,\q,\mathrm{old}),\end{equation} where 
\begin{equation}\label{eqn:main-term}M_{\p}^{\g}(\bk,\q)=2\frac{C_{\g}(\p)}{\sqrt{\norm(\p)}}\sum_{d=1}^{\infty}\frac{a_{d}^{\n\q}}{d}V_{\frac{1}{2}}\left(\frac{4^n\pi^{2n}\norm(\p)d^2}{\norm(\Dif_{F}^{2}\n\q)}\right),\end{equation}

\begin{align}\label{eqn:error-term}E_{\p}^{\g}(\bk,\q)&=2C\sum_{\{\a\}}\sum_{\nu\in(\a^{-1})^{+}/\ringO_{F}^{\times +}} \frac{C_{\g}(\nu\a)}{\sqrt{\norm(\nu\a)}}\sum_{d=1}^{\infty}\frac{a_{d}^{\n\q}}{d}V_{\frac{1}{2}}\left(\frac{4^n\pi^{2n}\norm(\nu\a)d^2}{\norm(\Dif_{F}^{2}\n\q)}\right)\\
&\hspace{.7in}\times\sum_{\substack{\c^{2}\sim\a\b\\c\in\c^{-1}\q\backslash\{0\}\\\epsilon\in\ringO_{F}^{\times+}/\ringO_{F}^{\times2}}}\frac{{\mathit Kl}(\epsilon\nu,\a; \xi,\b;c,\c)}{\norm(c\c)}\prod_{j=1}^nJ_{k_j-1}\left(\frac{4\pi\sqrt{\epsilon_j\nu_j\xi_j\left[\a\b\c^{-2}\right]_j}}{|c_j|}\right)\nonumber\end{align}
and
\begin{align}\label{eqn:old-forms}E_{\p}^{\g}(\bk,\q,\mathrm{old})&=2\sum_{\m\subset\ringO_{F}} \frac{C_{\g}(\m)}{\sqrt{\norm(\m)}}\sum_{d=1}^{\infty}\frac{a_{d}^{\n\q}}{d}V_{\frac{1}{2}}\left(\frac{4^n\pi^{2n}\norm(\m)d^2}{\norm(\Dif_{F}^{2}\n\q)}\right)\\
&\hspace{.7in}\times\sum_{\f\in H_{\bk}^{\mathrm{old}}(\q)}\frac{\Gamma(\bk-\boldsymbol{1})}{(4\pi)^{\bk-\boldsymbol{1}}|d_{F}|^{1/2}\left<\f,\f\right>_{S_{\bk}(\q)}}C_{\f}(\m)\overline{C_{\f}(\p)}.\nonumber\end{align}

In the following section, we prove that, as $\norm(\q)\rightarrow \infty$, we have
\[ M_\p^\g(\bk, \q)\sim \frac{C_{\g}(\p)}{\sqrt{\norm(\p)}}\gamma_{-1}(F)\prod_{\substack{\l|\n \\ \l \,:\,\mathrm{prime}}}(1-\norm(\l)^{-1})\log(\norm(\q)),\]
where as $E_\p^\g(\bk, \q)$ and $E_\p^\g(\bk,\q,\mathrm{old})$ are $O(1)$, which completes the proof of Proposition~\ref{main-proposition}.

\section{Asymptotic Formula for Harmonic Average}\label{sec:average}

\subsection{Main Term $M_{\p}^{\g}(\bk,\q)$}\label{sec:main-term}
In this section we establish an asymptotic estimate (as $\norm(\q)\to\infty$) for the main term given by (\ref{eqn:main-term}). We have 
\begin{align*}&\sum_{d=1}^{\infty}\frac{a_{d}^{\n\q}}{d}V_{1/2}\left(\frac{4^n\pi^{2n}\norm(\p)d^2}{\norm(\Dif_{F}^{2}\n\q)}\right)\\
&\hspace{.5in}=\frac{1}{2\pi i}\sum_{d=1}^{\infty}\frac{a_{d}^{\n\q}}{d}\int_{(3/2)}G(u)\left(\frac{4^n\pi^{2n}\norm(\p)d^2}{\norm(\Dif_{F}^{2}\n\q)}\right)^{-u}
\gamma\left(\frac12, u\right)\;\frac{du}{u}\\
&\hspace{.5in}=\frac{1}{2\pi i}\int_{(3/2)}G(u)\left(\frac{4^n\pi^{2n}\norm(\p)}{\norm(\Dif_{F}^{2}\n\q)}\right)^{-u}
\gamma\left(\frac12, u\right)\;\frac{du}{u},\end{align*}
where $\gamma(1/2, u)$ is as defined in (\ref{eqn:gamma}).
Setting $G(u)=1$ and shifting the contour of integration to $\Re(u)=-1/4$ give 
\begin{align}\label{eqn:diagonal}
&\sum_{d=1}^{\infty}\frac{a_{d}^{\n\q}}{d}V_{1/2}\left(\frac{4^n\pi^{2n}\norm(\p)d^2}{\norm(\Dif_{F}^{2}\n\q)}\right)\\
&\hspace{.5in}=\res_{u=0}\left(\left(\frac{4^n\pi^{2n}\norm(\p)}{\norm(\Dif_{F}^{2}\n\q)}\right)^{-u}
\gamma\left(\frac12, u\right)\frac{\zeta_{F}^{\n\q}(2u+1)}{u}\right) \nonumber \\
&\hspace{1.2in}
+\frac{1}{2\pi i}\int_{(-1/4)}\left(\frac{4^n\pi^{2n}\norm(\p)}{\norm(\Dif_{F}^{2}\n\q)}\right)^{-u}
\gamma\left(\frac12, u\right)\zeta_{F}^{\n\q}(2u+1)\;\frac{du}{u}.\nonumber \end{align}
It can be easily verified that the integral over the vertical line $\Re(u)=-1/4$ on the right hand side of (\ref{eqn:diagonal}) is $O\left(\norm(\q)^{-\frac{1}{4}}\right)$. 
As for the residue at $u=0$, we use the following standard Taylor series expansions 
\begin{eqnarray*}
\frac{\Gamma(a+u)}{\Gamma(a)}&=&1+\frac{\Gamma^\prime(a)}{\Gamma(a)}u+\cdots,\\
\left(\frac{4^n\pi^{2n}\norm(\p)}{\norm(\Dif_{F}^{2}\n\q)}\right)^{-u}&=&1-\log\left(\frac{4^n\pi^{2n}\norm(\p)}{\norm(\Dif_{F}^{2}\n\q)}\right)u+\cdots, \\
\zeta_F(2u+1)&=&\frac{\gamma_{-1}(F)}{2u}+\gamma_0(F)+\cdots,\end{eqnarray*}
along with the identity  $$\zeta_{F}^{\n\q}(2u+1)=\zeta_{F}(2u+1)(1-\norm(\q)^{-2u-1})\prod_{\substack{\l|\n \\ \l\,:\, \text{prime}}}(1-\norm(\l)^{-2u-1}),$$ 
to conclude that 
\begin{align*}
&\res_{u=0}\left(\left(\frac{4^n\pi^{2n}\norm(\p)}{\norm(\Dif_{F}^{2}\n\q)}\right)^{-u}\gamma\left(\frac12, u\right)\frac{\zeta_{F}^{\n\q}(2u+1)}{u}\right)\\
&\hspace{.5in}=\frac{\gamma_{-1}(F)}{2}\prod_{\substack{\l|\n \\ \l \,:\,\mathrm{prime}}}(1-\norm(\l)^{-1})\log(\norm(\q))+D_{\g}+O\left(\frac{\log(\norm(\q))}{\norm(\q)}\right).
\end{align*} 
We mention here that $D_{\g}$ is a constant independent of $\q$ and could be explicitly computed if need be. Therefore, $$M_{\p}^{\g}(\bk,\q)=\frac{C_{\g}(\p)}{\sqrt{\norm(\p)}}\gamma_{-1}(F)\prod_{\substack{\l|\n \\ \l \,:\,\mathrm{prime}}}(1-\norm(\l)^{-1})\log(\norm(\q))+O(1),\quad\quad\text{as}\; \norm(\q)\to\infty.$$

\subsection{Error Term $E_{\p}^{\g}(\bk,\q)$}\label{sec:error-term}
In order to give an asymptotic estimate for the error term $E_\p^\g(\bk,\q)$ given in (\ref{eqn:error-term}), it suffices to consider the expression\begin{align}\label{eqn:error}
E_{\p,\a}^\g(\bk,\q)&=\sum_{\nu\in(\a^{-1})^{+}/ \ringO_{F}^{\times +}} \frac{C_{\g}(\nu\a)}{\sqrt{\norm(\nu\a)}}\sum_{d=1}^{\infty}\frac{a_{d}}{d}V_{1/2}\left(\frac{4^n\pi^{2n}\norm(\nu\a)d^2}{\norm(\Dif_{F}^{2}\n\q)}\right)\\
&\nonumber \hspace{.5in} \times \sum_{c\in\c^{-1}\q\backslash\{0\}/\ringO_{F}^{\times +}}\sum_{\eta\in\ringO_{F}^{\times+}}\frac{{\mathit Kl}(\nu,\a; \xi,\b;c\eta,\c)}{|\norm(c)|}\prod_{j=1}^{n}J_{k_j-1}\left(\frac{4\pi\sqrt{\nu_j\xi_j\left[\a\b\c^{-2}\right]_j}}{\eta_{j}|c_j|}\right). \nonumber
\end{align} 
for any ideal class representative $\a$, while fixing an ideal class respresentative $\c$ such that $\c^{2}\sim\a\b$ and ignoring the (finite) sum over $\epsilon$. By Lemma~\ref{normlemmatrotabas}, we may assume that the representatives $\nu\in(\a^{-1})^{+}/ \ringO_{F}^{\times +}$  and $c\in\c^{-1}\q\backslash\{0\}/\ringO_{F}^{\times +}$ in (\ref{eqn:error}) satisfy \begin{equation}\label{eqn:norms}
\norm(\nu)^{1/n}\ll\nu_{j}\ll\norm(\nu)^{1/n}\quad \text{and} \quad |\norm(c)|^{1/n}\ll|c_{j}|\ll|\norm(c)|^{1/n},\quad \forall\; j\in\{1,\cdots,n\}.\end{equation}

We obtain an upper bound for $E_{\p,\a}^{\g}(k,\q)$ as $\norm(\q)\to\infty$ by applying the estimates for the $J$-Bessel function and the Kloosterman sum given in (\ref{J-Bessel0}) and (\ref{weilbound}). In particular, the values of the $J$-Bessel function in (\ref{eqn:error}) are bounded as follows. We take $\delta_{j}=0$ if $\eta_{j}\geq1$, and otherwise $\delta_{j}=\delta$ for some fixed (sufficiently small) $\delta>0$. With this choice of ${\boldsymbol{\delta}}=(\delta_{j})$, we have
\begin{align}\label{eqn:J-Bessel1}\prod_{j=1}^nJ_{k_j-1}\left(\frac{4\pi\sqrt{\nu_j\xi_j\left[\a\b\c^{-2}\right]_j}}{\eta_{j}|c_j|}\right)&\ll\prod_{j=1}^{n}\left(\frac{\sqrt{\nu_{j}\xi_{j}\left[\a\b\c^{-2}\right]_{j}}}{\eta_{j}|c_j|}\right)^{1-\delta_{j}}\\&=\left(\frac{\sqrt{\boldsymbol{\nu\xi\left[\a\b\c^{-2}\right]}}}{\boldsymbol{\eta|c|}}\right)^{\boldsymbol{1}-\boldsymbol{\delta}}.\nonumber\end{align} 
This allows us to control the internal sum in (\ref{eqn:error}) over all $\eta\in\ringO_{F}^{\times +}$ since (thanks again to the work of Luo \cite[p.\:136]{luo1}) 
\begin{equation}\label{eqn:unit}
\sum_{\eta\in\ringO_{F}^{\times+}}\prod_{\eta_{j}<1}\eta_{j}^{\delta}<\infty.
\end{equation}
Upon applying the bounds (\ref{weilbound}) and (\ref{eqn:J-Bessel1}), we get
\begin{align*}E_{\p,\a}^{\g}(\bk,\q)&\ll\sum_{\nu\in(\a^{-1})^{+}/\ringO_{F}^{\times +}} \frac{|C_{\g}(\nu\a)|}{\sqrt{\norm(\nu\a)}}\sqrt{\boldsymbol{\nu}}^{\boldsymbol{1}-\boldsymbol{\delta}}\sum_{d=1}^{\infty}\frac{a_{d}^{\n\q}}{d}\left|V_{1/2}\left(\frac{4^n\pi^{2n}\norm(\nu\a)d^2}{\norm(\Dif_{F}^{2}\n\q)}\right)\right|\\
&\hspace{.7in}\times\sum_{\eta\in\ringO_{F}^{\times+}}\boldsymbol{\eta}^{\boldsymbol{\delta}}\sum_{c\in\c^{-1}\q\backslash\{0\}/\ringO_{F}^{\times+}}|\boldsymbol{c}|^{\boldsymbol{\delta}-\boldsymbol{1}}\frac{\norm\left((\nu\a,\xi\b,c\c)\right)^{\frac{1}{2}}\tau(c\c)}{\sqrt{\norm(c\c)}}.\end{align*} 
Using the estimates (\ref{eqn:tau}), (\ref{eqn:norms}) and (\ref{eqn:unit}), we see that
\[E_{\p,\a}^{\g}(\bk,\q)\ll\sum_{\nu\in(\a^{-1})^{+}/\ringO_{F}^{\times +}} |C_{\g}(\nu\a)|\sum_{d=1}^{\infty}\frac{a_{d}^{\n\q}}{d}\left|V_{1/2}\left(\frac{4^n\pi^{2n}\norm(\nu\a)d^2}{\norm(\Dif_{F}^{2}\n\q)}\right)\right|\sum_{\c\subset\q}\frac{\norm\left((\nu\a,\xi\b,\c)\right)^{\frac{1}{2}}}{\norm(\c)^{\frac{3}{2}-\delta}}.\] 
On the other hand, we have (see \cite[p.\:228]{trotabas})
$$\sum_{\c\subset\q}\frac{\norm\left((\nu\a,\xi\b,\c)\right)^{\frac{1}{2}}}{\norm(\c)^{\frac{3}{2}-\delta}}\ll\frac{\norm\left((\nu\a,\xi\b,\q)\right)^{\frac{1}{2}}}{\norm(\q)^{\frac{3}{2}-\delta}}\norm\left((\nu\a,\xi\b)\right)^{\delta}.$$ 
Hence, 
\begin{align*}E_{\p,\a}^{\g}(\bk,\q)&\ll\norm(\q)^{-\frac{3}{2}+\delta}\sum_{\nu\in(\a^{-1})^{+}/\ringO_{F}^{\times +}} |C_{\g}(\nu\a)|\sum_{d=1}^{\infty}\frac{a_{d}^{\n\q}}{d}\left|V_{1/2}\left(\frac{4^n\pi^{2n}\norm(\nu\a)d^2}{\norm(\Dif_{F}^{2}\n\q)}\right)\right|\\
&\ll\norm(\q)^{-\frac{3}{2}+\delta}\left(\sum_{\substack{\nu\in(\a^{-1})^{+}/\ringO_{F}^{\times +}\\ \norm(\nu\a)\ll\norm(\q)}} |C_{\g}(\nu\a)|+\sum_{\substack{\nu\in(\a^{-1})^{+}/\ringO_{F}^{\times +}\\ \norm(\nu\a)\gg\norm(\q)}}|C_{\g}(\nu\a)|\left(\frac{\norm(\nu\a)}{\norm(\q)}\right)^{-A}\right),\end{align*}
where the last inequality is obtained by using the estimates (\ref{eqn:V-estimate1}) and (\ref{eqn:V-estimate2}). Given $\epsilon>0$, it follows from the Ramanujan bound (\ref{eqn:ramanujan-bound}) that both sums are $O\left(\norm(\q)^{1+\epsilon}\right)$. Therefore, \[E_{\p,\a}^{\g}(\bk,\q)= O\left(\norm(\q)^{-\frac{1}{2}+\delta+\epsilon}\right),\quad\quad \text{as}\;\norm(\q)\to\infty.\]

\subsection{Contribution of Old Forms $E_{\p}^{\g}(\bk,\q,\mathrm{old})$}\label{sec:old-forms}
  Let us first describe an orthogonal basis for the space of oldforms in $S_{\bk}(\q)$ following the treatment in \cite[Section~11]{trotabas}. For $\f\in \Pi_{\bk}(\ringO_{F})$ and $g\in GL_{2}(\mathbb{A}_{F})$, we set $$\f_{\q}(g)=\left(\frac{\norm(\q)}{\rho_{\f}(\q)}\right)^{\frac{1}{2}}\sum_{\d\e=\q}\frac{\mu(\d)C_{\f}(\d)}{\psi(\d)\norm(\d)}\norm(\e)^{-\frac{1}{2}}\f\left(g\left[ {\begin{array}{cc}
   \mathrm{id}(\e)^{-1} & 0 \\
   0 & 1 \\
  \end{array} } \right]\right),$$ where $\displaystyle{\rho_{\f}(\q)=\prod_{\l|\q}\left(1-\norm(\l)\left(\frac{C_{\f}(\l)}{\norm(\l)+1}\right)^{2}\right)}$, $\displaystyle{\psi(\d)=\prod_{\l|\d}(1+\norm(\l)^{-1})}$, $\mu(\d)$ is the generalized  Mobius function for number fields and $\mathrm{id}(\e)$ is the idele of $F$ associated with the ideal $\e$. The set $\{\f,\f_{\q}\}_{\f\in \Pi_{k}(\ringO_{F})}$ is an orthogonal basis for $S_{k}^{\mathrm{old}}(\q)$ with $$\left<\f_{\q},\f_{\q}\right>_{S_{\bk}(\q)}=\left<\f,\f\right>_{S_{\bk}(\q)}=(\norm(\q)+1)\left<\f,\f\right>_{S_{\bk}(\ringO_{F})}.$$ Moreover, the Fourier coefficients of $\f_{\q}$ are given by 
\begin{equation}\label{eqn:fourier-coefficients}C_{\f_{\q}}(\m)=\left(\norm(\q)(1-\norm(\q)^{-2})(1+\norm(\q)^{-1})L_{\q}(\mathrm{sym}^{2}\f,1)\right)^{\frac{1}{2}}\left(-\frac{C_{\f}(\q)C_{\f}(\m)}{\norm(\q)+1}+C_{\f}(\m\q^{-1})\1_{\q|\m}\right),\end{equation} 
where $L_{\q}(\mathrm{sym}^{2}\f,1)$ is the Euler factor at $\q$ of the symmetric square $L$-function of $\f$. 

The rest of this section is devoted to show that the contribution of the oldforms given by $E_{\p}^{\g}(\bk,\q,\mathrm{old})$ in (\ref{eqn:old-forms}) satisfies \begin{equation}\label{eqn:final-claim}E_{\p}^{\g}(\bk,\q,\mathrm{old})\ll\norm(\q)^{-\frac{1}{2}+\epsilon},\quad\quad\text{as}\; \norm(\q)\to\infty.\end{equation} In view of the above discussion, we can write $E_{\p}^{\g}(\bk,\q,\mathrm{old})$ as
$$E_{\p}^{\g}(\bk,\q,\mathrm{old})=\frac{2\Gamma(\bk-\boldsymbol{1})}{(4\pi)^{\bk-\boldsymbol{1}}|d_{F}|^{1/2}}(E_{1}+E_{2}),$$
where $$E_{1}=\sum_{\m\subset\ringO_{F}} \frac{C_{\g}(\m)}{\sqrt{\norm(\m)}}\sum_{d=1}^{\infty}\frac{a_{d}^{\n\q}}{d}V_{\frac{1}{2}}\left(\frac{4^n\pi^{2n}\norm(\m)d^2}{\norm(\Dif_{F}^{2}\n\q)}\right)\sum_{\f\in \Pi_{k}(\ringO_{F})}\frac{C_{\f}(\m)C_{\f}(\p)}{\left<\f,\f\right>_{S_{\bk}(\q)}}$$ and 
$$E_{2}=\sum_{\m\subset\ringO_{F}} \frac{C_{\g}(\m)}{\sqrt{\norm(\m)}}\sum_{d=1}^{\infty}\frac{a_{d}^{\n\q}}{d}V_{\frac{1}{2}}\left(\frac{4^n\pi^{2n}\norm(\m)d^2}{\norm(\Dif_{F}^{2}\n\q)}\right)\sum_{\f\in \Pi_{\bk}(\ringO_{F})}\frac{C_{\f_{\q}}(\m)C_{\f_{\q}}(\p)}{\left<\f_{\q},\f_{\q}\right>_{S_{\bk}(\q)}}.$$
 Notice that 
\begin{equation}\label{eqn:level1} \sum_{\f\in \Pi_{\bk}(\ringO_{F})}\frac{C_{\f}(\m)C_{\f}(\p)}{\left<\f,\f\right>_{S_{\bk}(\q)}}\leq\frac{1}{1+\norm(\q)}\sum_{\f\in \Pi_{\bk}(\ringO_{F})}\frac{\left|C_{\f}(\m)C_{\f}(\p)\right|}{\left<\f,\f\right>_{S_{\bk}(\ringO_{F})}}\ll\norm(\q)^{-1}\norm(\m)^{\epsilon}.
\end{equation}
Using (\ref{eqn:level1}) and the estimates (\ref{eqn:V-estimate1}) and (\ref{eqn:V-estimate2}), we obtain
\begin{align*}E_{1}&\ll\norm(\q)^{-1}\sum_{\m\subset\ringO_{F}} |C_{\g}(\m)|\norm(\m)^{\epsilon-\frac{1}{2}}\sum_{d=1}^{\infty}\frac{a_{d}^{\n\q}}{d}\left|V_{1/2}\left(\frac{4^n\pi^{2n}\norm(\m)d^2}{\norm(\Dif_{F}^{2}\n\q)}\right)\right|\\
&\ll\norm(\q)^{-1}\left(\sum_{\substack{\m\subset\ringO_{F}\\\norm(\m)\ll\norm(\q)}} |C_{\g}(\m)|\norm(\m)^{\epsilon-\frac{1}{2}}+\sum_{\substack{\m\subset\ringO_{F}\\\norm(\m)\gg\norm(\q)}}|C_{\g}(\m)|\norm(\m)^{\epsilon-\frac{1}{2}}\left(\frac{\norm(\m)}{\norm(\q)}\right)^{-A}\right)\\
&\ll\norm(\q)^{-\frac{1}{2}+\epsilon},
\end{align*} 
where the last inequality follows from (\ref{eqn:ramanujan-bound}).

Finally, we consider the contribution of the forms $\f_{\q}$ for $\f\in\Pi_{\bk}(\ringO_{F})$. We apply the identity (\ref{eqn:fourier-coefficients}) along with the bound $L_{\q}(\mathrm{sym}^{2}\f,1)\ll1$ (as $\norm(q)\to\infty$) to get
\[\sum_{\f\in \Pi_{\bk}(\ringO_{F})}\frac{C_{\f_{\q}}(\m)C_{\f_{\q}}(\p)}{\left<\f_{\q},\f_{\q}\right>_{S_{\bk}(\q)}}\ll\sum_{\f\in \Pi_{\bk}(\ringO_{F})}\frac{C_{\f}(\q)^{2}|C_{\f}(\m)C_{\f}(\p)|}{(\norm(\q)+1)^2\left<\f,\f\right>_{S_{\bk}(\ringO_{F})}}+\1_{\q|\m}\sum_{\f\in \Pi_{\bk}(\ringO_{F})}\frac{|C_{\f}(\m\q^{-1})C_{\f}(\q)C_{\f}(\p)|}{(\norm(\q)+1)\left<\f,\f\right>_{S_{\bk}(\ringO_{F})}}.\]
Hence, we have \begin{align}\label{eqn:l-divides-m}
E_{2}&\ll\sum_{\m\subset\ringO_{F}} \frac{|C_{\g}(\m)|}{\sqrt{\norm(\m)}}\sum_{d=1}^{\infty}\frac{a_{d}^{\n\q}}{d}\left|V_{1/2}\left(\frac{4^n\pi^{2n}\norm(\m)d^2}{\norm(\Dif_{F}^{2}\n\q)}\right)\right|\sum_{\f\in \Pi_{\bk}(\ringO_{F})}\frac{C_{\f}(\q)^{2}|C_{\f}(\m)C_{\f}(\p)|}{(\norm(\q)+1)^2\left<\f,\f\right>_{S_{\bk}(\ringO_{F})}}\nonumber\\
&\hspace{.5in}+\sum_{\substack{\m\subset\ringO_{F}\\\q|\m}} \frac{|C_{\g}(\m)|}{\sqrt{\norm(\m)}}\sum_{d=1}^{\infty}\frac{a_{d}^{\n\q}}{d}\left|V_{1/2}\left(\frac{4^n\pi^{2n}\norm(\m)d^2}{\norm(\Dif_{F}^{2}\n\q)}\right)\right|\sum_{\f\in \Pi_{\bk}(\ringO_{F})}\frac{|C_{\f}(\m\q^{-1})C_{\f}(\q)C_{\f}(\p)|}{(\norm(\q)+1)\left<\f,\f\right>_{S_{\bk}(\ringO_{F})}}.\end{align}
After writing the sum (\ref{eqn:l-divides-m}) as
$$\sum_{\m\subset\ringO_{F}} \frac{|C_{\g}(\m\q)|}{\sqrt{\norm(\m\q)}}\sum_{d=1}^{\infty}\frac{a_{d}^{\n\q}}{d}\left|V_{1/2}\left(\frac{4^n\pi^{2n}\norm(\m)d^2}{\norm(\n)}\right)\right|\sum_{\f\in \Pi_{\bk}(\ringO_{F})}\frac{|C_{\f}(\m)C_{\f}(\q)C_{\f}(\p)|}{(\norm(\q)+1)\left<\f,\f\right>_{S_{\bk}(\ringO_{F})}},$$ we get
\begin{align*}E_{2}&\ll\frac{\norm(\q)^{\epsilon}}{(\norm(\q)+1)^{2}}\sum_{d=1}^{\infty}\frac{a_{d}^{\n\q}}{d}\sum_{\m\subset\ringO_{F}} |C_{\g}(\m)|\norm(\m)^{\epsilon-\frac{1}{2}}
\left|V_{1/2}\left(\frac{4^n\pi^{2n}\norm(\m)d^2}{\norm(\Dif_{F}^{2}\n\q)}\right)\right|\\
&\hspace{.5in}+\frac{1}{\norm(\q)+1}\sum_{d=1}^{\infty}\frac{a_{d}^{\n\q}}{d}\sum_{\m\subset\ringO_{F}} |C_{\g}(\m\q)|\norm(\m\q)^{\epsilon-\frac{1}{2}}\left|V_{1/2}\left(\frac{4^n\pi^{2n}\norm(\m)d^2}{\norm(\n)}\right)\right|.\end{align*}
Once again we use the estimates (\ref{eqn:V-estimate1}) and (\ref{eqn:V-estimate2}) to get
\begin{align*}&\sum_{d=1}^{\infty}\frac{a_{d}^{\n\q}}{d}\sum_{\m\subset\ringO_{F}} |C_{\g}(\m)|\norm(\m)^{\epsilon-\frac{1}{2}}\left|V_{1/2}\left(\frac{4^n\pi^{2n}\norm(\m)d^2}{\norm(\Dif_{F}^{2}\n\q)}\right)\right|\\
&\hspace{.3in}\ll\sum_{d=1}^{\infty}\frac{a_{d}^{\n\q}}{d}\left(\sum_{\substack{\m\subset\ringO_{F}\\\norm(\m)\ll d^{-2}\norm(\q)}} |C_{\g}(\m)|\norm(\m)^{\epsilon-\frac{1}{2}}+\sum_{\substack{\m\subset\ringO_{F}\\\norm(\m)\gg d^{-2}\norm(\q)}} |C_{\g}(\m)|\norm(\m)^{\epsilon-\frac{1}{2}}\left(\frac{\norm(\m)d^2}{\norm(\q)}\right)^{-A}\right)\\
&\hspace{.3in}\ll\norm(\q)^{\frac{1}{2}+\epsilon},\end{align*} 
and 
\begin{align*}&\sum_{d=1}^{\infty}\frac{a_{d}^{\n\q}}{d}\sum_{\m\subset\ringO_{F}} |C_{\g}(\m\q)|\norm(\m\q)^{\epsilon-\frac{1}{2}}\left|V_{1/2}\left(\frac{4^n\pi^{2n}\norm(\m)d^2}{\norm(\n)}\right)\right|\\
&\hspace{.3in}\ll\sum_{d=1}^{\infty}\frac{a_{d}^{\n\q}}{d}\left(\sum_{\substack{\m\subset\ringO_{F}\\\norm(\m)\ll d^{-2}}} |C_{\g}(\m\q)|\norm(\m\q)^{\epsilon-\frac{1}{2}}+\sum_{\substack{\m\subset\ringO_{F}\\\norm(\m)\gg d^{-2}}} |C_{\g}(\m\q)|\norm(\m\q)^{\epsilon-\frac{1}{2}}(\norm(\m)d^{2})^{-A}\right)\\
&\hspace{.3in}\ll\norm(\q)^{\epsilon-\frac{1}{2}}.\end{align*} Therefore, we have $E_{2}\ll\norm(\q)^{\epsilon-\frac{3}{2}}$, as $\norm(\q)\to\infty$, which concludes the proof of (\ref{eqn:final-claim}).

\section*{Acknowledgements}
The authors would like to express their gratitude to Amir Akbary for reading the manuscript and providing valuable suggestions which improved the exposition of the paper. The authors would also like to thank Wenzhi Luo, M. Ram Murty and John Voight for useful discussions related to this work. 
\bibliographystyle{amsplain}
\bibliography{references}

\end{document}